\documentclass[11pt, reqno, twoside, makeidx]{amsart}
\usepackage[margin=2.7cm, marginpar=1cm]{geometry}

\usepackage{wrapfig}
\usepackage{marginnote}
\usepackage{hyperref}
\usepackage[english]{babel}
\usepackage[utf8]{inputenc}
\usepackage[T1]{fontenc}
\usepackage{amsmath,amsthm,amssymb,amsfonts}
\numberwithin{equation}{section}
\usepackage{enumerate}
\usepackage{faktor}
\usepackage{dsfont}
\usepackage{scalerel,stackengine}
\usepackage{textcomp}
\usepackage{graphicx}
\usepackage{extarrows}
\usepackage{nicematrix}
\usepackage{epigraph}
\usepackage{graphicx}
\usepackage{subcaption}
\usepackage{sidecap}
\usepackage{indentfirst}
\usepackage{tikz}
\usepackage{tikz-cd}
\usepackage{tkz-euclide}
\usetikzlibrary{shapes.misc,arrows,matrix,patterns,decorations.markings,positioning}
\tikzset{cross/.style={cross out, draw=black, minimum size=2*(#1-\pgflinewidth), inner sep=0pt, outer sep=0pt},
cross/.default={4.5pt}}
\usepackage{pgfplots}
\usepackage{caption}
\usepackage{float}
\usepackage{color}
\usepackage{epstopdf}
\usepackage{epsfig}
\usepackage{stmaryrd}
\usepackage{color}
\usepackage{geometry}
\usepackage{multicol}
\usepackage{verbatim}
\usepackage{wrapfig}
\usepackage{float}

\DeclareMathOperator{\Ker}{Ker }
\DeclareMathOperator{\Imm}{Im }
\DeclareMathOperator{\tb}{tb}

\DeclareMathOperator{\Spin}{Spin}

\renewcommand{\geq}{\geqslant}
\renewcommand{\leq}{\leqslant} 
\renewcommand{\epsilon}{\varepsilon}

\newcommand{\Q}{\mathbb{Q}}

\newcommand{\s}{\mathfrak{s}}

\newtheorem{teo}{Theorem}[section]
\newtheorem*{teo*}{Theorem}
\newtheorem{lemma}[teo]{Lemma}
\newtheorem{prop}[teo]{Proposition}
\newtheorem*{prop*}{Proposition}

\newtheorem{cor}[teo]{Corollary}

\usepackage{xpatch}
\makeatletter
\xpatchcmd{\@thm}{\thm@headpunct{.}}{\thm@headpunct{}}{}{}
\makeatother

\pgfplotsset{compat=1.18}
\begin{document}
\title[Hat and plus version of the contact invariant are not equivalent]{The hat and plus version of the Heegaard Floer contact invariant are not equivalent}
\author{Alberto Cavallo and Irena Matkovi\v c}
\address{HUN-REN Alfr\'ed R\'enyi Insitute of Mathematics, Budapest 1053, Hungary}
\email{acavallo@impan.pl}
\address{Uppsala Universitet, Uppsala 751 06, Sweden}
\email{irma6504@student.uu.se}


\begin{abstract}
 We advance Matkovi\v c ideas, originally applied to complete the classification of tight structures on small Seifert fibred $L$-spaces, to show the existence of contact structures on Brieskorn spheres which are tight and zero-twisting. This uncovers a phenomenon that has never appeared in literature before: namely, that a contact structure $\xi$ on a 3-manifold can be such that $\widehat c(\xi)$ is non-vanishing, but $c^+(\xi)$ is zero. 
\end{abstract}

\maketitle

\thispagestyle{empty}

\section{Introduction}
In \cite{CM-negative} we classified all the fillable structures on a Seifert fibred space with negative-definite standard graph $\Gamma$; moreover, we showed that they all appear as the boundary of the same compact 4-manifold, which is obtained by blowing down $\Gamma$ completely. Later on, in \cite{CM} we classified all the negative-twisting structures when the standard graph is indefinite.
Part of these results can be summarised in the following theorem; we refer to \cite{CM-negative,CM} for the definition of negative- and zero-twisting tight structures, and the notation. 
\begin{teo}[Cavallo-Matkovi\v c]
 \label{teo:old} 
 Suppose that $M=M(e_0;r_1,...,r_n)$ is a Seifert fibred space, and that $\xi$ is a contact structure on $M$. If $M$ has negative-definite standard graph then the following are equivalent:
 \begin{itemize}
     \item the structure $\xi$ is Stein fillable;
     \item the structure $\xi$ is symplectically fillable;
     \item the structure $\xi$ has non-vanishing contact invariant $c^+(\xi)$;
     \item the structure $\xi$ is negative-twisting.
 \end{itemize}
 If $M$ has indefinite standard graph, is not an $L$-space, and $n=3$ then the following are equivalent:
 \begin{itemize}
     \item the structure $\xi$ is symplectically fillable;
     \item the structure $\xi$ has non-vanishing contact invariant $c^+(\xi)\in HF_\emph{red}(-M)$;
     \item the structure $\xi$ is negative-twisting.
 \end{itemize}
\end{teo}
When the graph is negative-definite the only tight structures possibly left out of our classification would necessarily  have vanishing invariant $c^+$, and be zero-twisting. Generalising the techniques used by Matkovi\v c in \cite{Irena}, which together with \cite{GLS,GLS=>,Wu} led to the classification of contact structures on small Seifert fibred $L$-spaces with indefinite standard graph, that is with $e_0\geq-1$ according to our notation, we show that Theorem \ref{teo:old} cannot be strengthened. 
\begin{teo}
 \label{teo:main} 
 All the $3$-manifolds in the infinite family of Seifert fibred spaces $S^3_{k}(T_{8,13})$ for $k\leq35$ carry a zero-twisting tight contact structure $\xi_k$ such that $\widehat c(\xi_k)\neq0$, see Figure \ref{Structures}. Furthermore, they provide the first examples of:
 \begin{enumerate}
     \item contact $3$-manifolds with non-vanishing $\widehat c$, but $c^+$ equal to zero;
     \item zero-twisting structures such that their $d_3$-invariant is different from the correction term;
     \item tight structures on a small Seifert fibred space that is not an $L$-space which are not (weakly) symplectically fillable.
 \end{enumerate}
\end{teo}

The manifolds in Theorem \ref{teo:main} are also one of the first examples of Seifert fibred spaces which admit both negative- and zero-twisting structures compatible with the same orientation. However, the existence of these has already been known to Lisca and Stipsicz, as can be understood from the comment below \cite[Theorem 1.2]{LS}. In fact, they mention a zero-twisting tight structure on $M(-1;\frac{5}{12}, \frac{1}{3}, \frac{1}{3})$ whose $d_3$-invariant is equal to the correction term; hence, with non-zero $c^+$. According to a result of Matkovi\v c \cite[Theorem 1.1]{Irena(f)} this structure is also not fillable.

Note that when $k\leq-1$ the manifolds above have negative-definite standard graph, while only when $k\geq83$ they are $L$-spaces. These structures appear even on Brieskorn spheres as we have that $\Sigma(8,13,105)\simeq S^3_{-1}(T_{8,13})$.

\begin{cor}
 \label{cor:main}
 Suppose that $M$ is a Brieskorn sphere among $\Sigma(8,13,105)$ and $-\Sigma(8,13,103)$. Then there exist zero-twisting structures on $M$ that are not fillable, and are therefore not included in our classification in Theorem \ref{teo:old}. 
\end{cor} 

\subsection*{Acknowledgements} {\smaller[1] We thank the Matematiska institutionen at Uppsala universitet for their friendly hospitality. A.C. has been partially supported by the HORIZON-ERC-2023-ADG 101141468 KnotSurf4d project.}

\section{Zero-twisting structures on small Seifert fibred spaces}
We start our discussion here by recalling the seminal result of Lisca and Stispicz about zero-twisting structures. They showed in \cite{LS} that every such structure on the space $M(-1;r_1,r_2,r_3)$ with $r_i\in\Q\:\cap(0,1)$ has a contact surgery presentation which consists of an appropriate Legendrian link whose components are a $T_{5,-5}$ and some unknots determined as follows. 

Let $r_i=[m^i_1,...,m^i_{k_i}]$ for $i=1,2,3$ be the negative continued fraction of $r_i$, so that each $m^i_j$ is the framing of a vertex of the standard graph; then there are three chains of unknots, attached to the last three components of $T_{5,-5}$, and the $i$-th chain contains $k_i-1$ unknots. The $\tb$-number of each component of $T_{5,-5}$ is given by the vector $(-1,-1,m^1_1,m^2_1,m^3_1)$, while on the three legs it is given by $(m^i_2+1,...,m^i_{k_i}+1)$ for $i=1,2,3$. The structure is then obtained by performing contact $(+1)$-surgery on the first two components of $T_{5,-5}$, and contact $(-1)$-surgery on all the other Legendrian knots. An example of this construction can be seen in Figure \ref{Structures}.

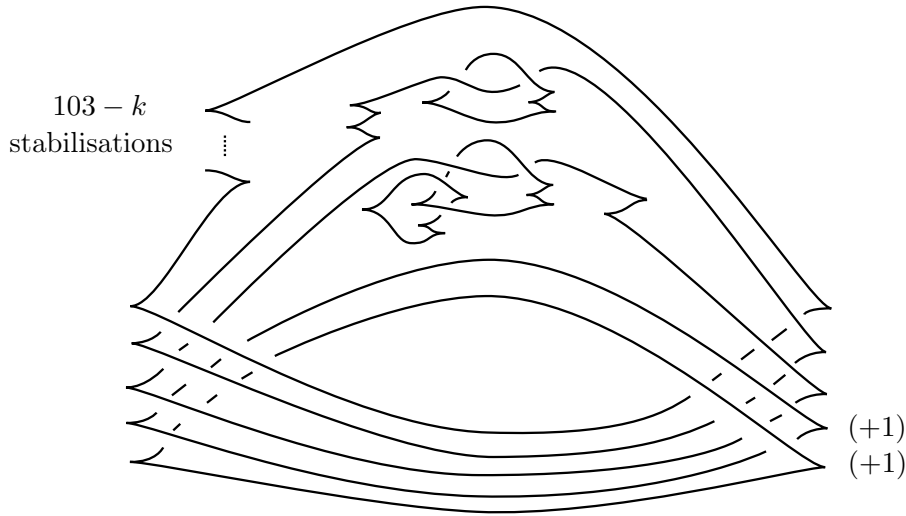
\begin{figure}[ht]
 \centering
  \def\svgwidth{12cm}
\begingroup%
  \makeatletter%
  \providecommand\color[2][]{%
    \errmessage{(Inkscape) Color is used for the text in Inkscape, but the package 'color.sty' is not loaded}%
    \renewcommand\color[2][]{}%
  }%
  \providecommand\transparent[1]{%
    \errmessage{(Inkscape) Transparency is used (non-zero) for the text in Inkscape, but the package 'transparent.sty' is not loaded}%
    \renewcommand\transparent[1]{}%
  }%
  \providecommand\rotatebox[2]{#2}%
  \newcommand*\fsize{\dimexpr\f@size pt\relax}%
  \newcommand*\lineheight[1]{\fontsize{\fsize}{#1\fsize}\selectfont}%
  \ifx\svgwidth\undefined%
    \setlength{\unitlength}{1482.20099502bp}%
    \ifx\svgscale\undefined%
      \relax%
    \else%
      \setlength{\unitlength}{\unitlength * \real{\svgscale}}%
    \fi%
  \else%
    \setlength{\unitlength}{\svgwidth}%
  \fi%
  \global\let\svgwidth\undefined%
  \global\let\svgscale\undefined%
  \makeatother%
  \begin{picture}(1,0.56285895)%
    \lineheight{1}%
    \setlength\tabcolsep{0pt}%
    \put(0,0){\includegraphics[width=\unitlength,page=1]{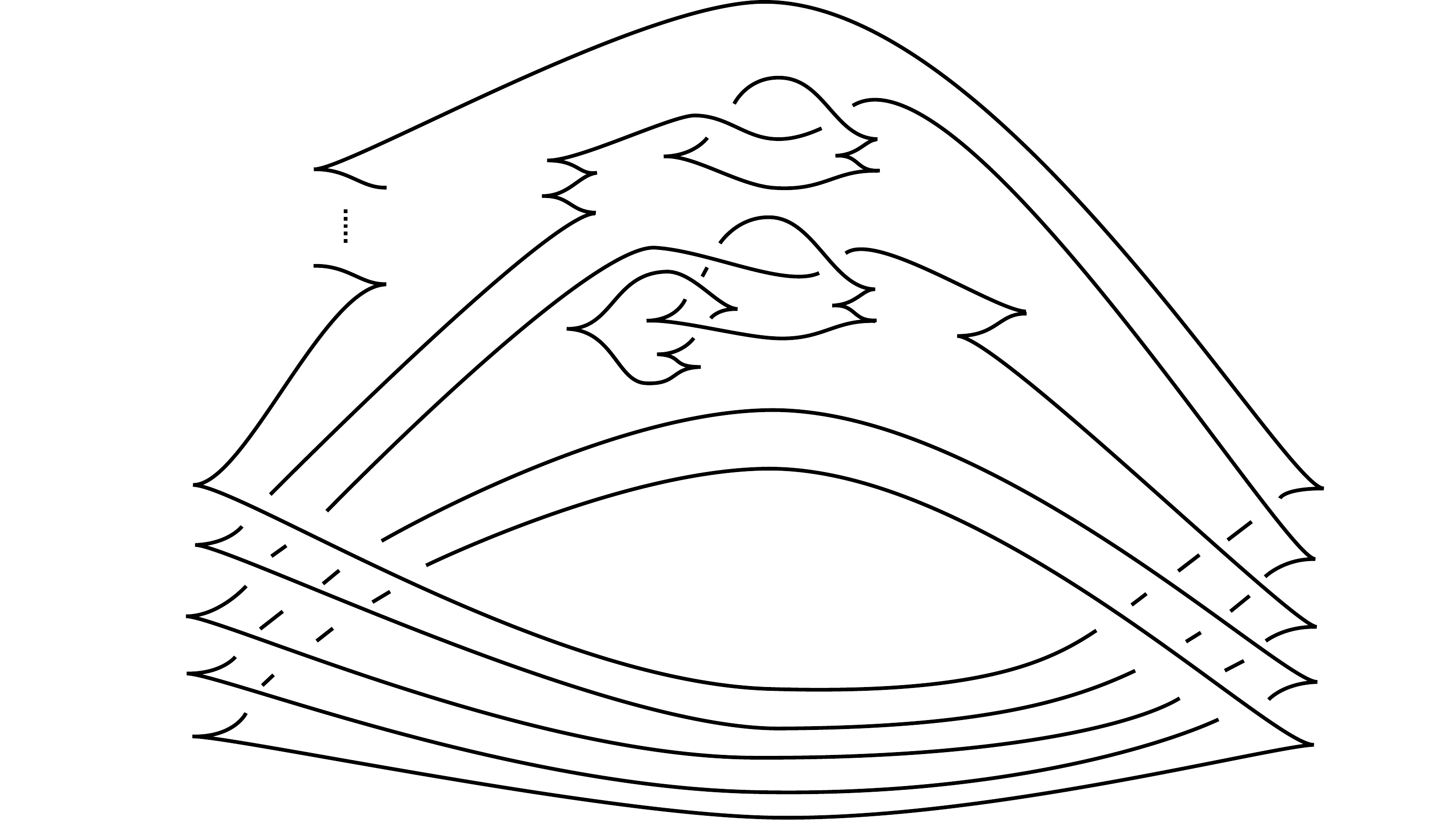}}%
    \put(0.04625657,0.44178385){\color[rgb]{0,0,0}\makebox(0,0)[lt]{\lineheight{1.25}\smash{\begin{tabular}[t]{l}$103-k$\end{tabular}}}}%
    \put(-0.00127886,0.40103922){\color[rgb]{0,0,0}\makebox(0,0)[lt]{\lineheight{1.25}\smash{\begin{tabular}[t]{l}$\text{stabilisations}$\end{tabular}}}}%
    \put(0.92760203,0.04549365){\color[rgb]{0,0,0}\makebox(0,0)[lt]{\lineheight{1.25}\smash{\begin{tabular}[t]{l}$(+1)$\end{tabular}}}}%
    \put(0.92690785,0.08570153){\color[rgb]{0,0,0}\makebox(0,0)[lt]{\lineheight{1.25}\smash{\begin{tabular}[t]{l}$(+1)$\end{tabular}}}}%
  \end{picture}%
\endgroup%

 \caption{\smaller[1]{A contact surgery presentation of $(S^3_k(T_{8,13}),\xi_k)$; the six unmarked Legendrian knots are contact $(-1)$-surgeries. When $k\leq35$ we have that $\widehat c(\xi_k)$ is non-vanishing while $c^+(\xi_k)=0$; hence, the structure $\xi_k$ is tight and zero-twisting but not fillable.}}
 \label{Structures}
 \end{figure}

 \begin{figure}[ht]
 \centering
  \def\svgwidth{7cm}
        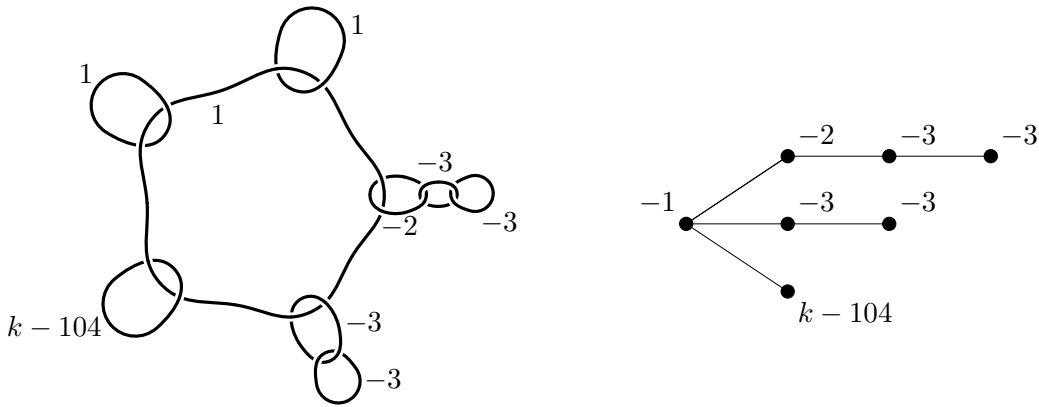
        \hspace{1cm}
        \begin{tikzpicture}[scale=0.9]
    \tkzDefPoints{0/0/A, 1.5/1/B, 1.5/-1/D, 1.5/0/C, 3/0/H, 3/1/N, 4.5/1/O, 0/-3/X}  
    \tkzDrawSegment(A,B)\tkzDrawSegment(A,D)\tkzDrawSegment(B,A)\tkzDrawSegment(A,C)\tkzDrawSegment(C,H)
    \tkzDrawSegment(B,N)\tkzDrawSegment(N,O)
    \tkzDrawPoints[fill,black,size=5](A,B,C,D,H,N,O)\tkzDrawPoints[fill,white,size=5](X)
     \tkzLabelPoint[above left](A){$-1$} 
     \tkzLabelPoint[above right](B){$-2$}\tkzLabelPoint[above right](H){$-3$}\tkzLabelPoint[above right](C){$-3$}
     \tkzLabelPoint[below right](D){$k-104$}\tkzLabelPoint[above right](N){$-3$}\tkzLabelPoint[above right](O){$-3$}
       \end{tikzpicture}
 \caption{\smaller[1]{Smooth surgery presentation (left) of $S^3_k(T_{8,13})\simeq M(-1;\frac{3}{8},\frac{8}{13},\frac{1}{104-k})$ corresponding to the surgery in Figure \ref{Structures}, after blowing up, and its standard graph $G$ (right).}}
 \label{Structures2}
 \end{figure}

In a similar way in \cite{GLS=>} it is also shown that the zero-twisting structures on $M(0,r_1,r_2,r_3)$ have a contact surgery presentation as the one described above, but without one of the two $(+1)$-surgeries. By modifying these constructions we can describe all the zero-twisting structures on $M(e_0,r_1,r_2,r_3)$ with $e_0\geq-1$; we do not give details about this generalisation as it is not relevant in this paper, but refer to \cite{GLS=>} for details.

There is not yet a complete classification of zero-twisting structures on these small Seifert fibred spaces in general, but either when the structures are fillable, or the manifold is an $L$-space, this has been done by Matkovi\v c in \cite{Irena,Irena(f)}. We describe the tools involved in such a classification (when $e_0=-1$) before stating the result. Let us consider the vector in $\text{Char}(G,\s_\xi)$ whose first coordinate (the one corresponding to the centre) is $1$, and the others are given by the rotation numbers in the chains of unknots and by one less than the rotation numbers of the last three components of $T_{5,-5}$ in the surgery presentation. 

Motivated by the work of Lisca and Stipsicz \cite{LSexist}, we describe an algorithm in \cite{Irena} that has the above vector as an input while it has as an output a vector $C\in\text{Char}(G^*,\s_\xi)$; we briefly sketch it here. By \cite[Lemma 3.2]{LSexist} the star-shaped plumbed $4$-manifolds $P_{G}$ embeds into some $R=\mathbb CP^2\#N\overline{\mathbb CP}^2$ in the way that its complement is diffeomorphic to $P_{G^*}$. More precisely, the configurations of the two intersection graphs $G$ and $G^*$ are obtained by blowing-up the initial lines $l_1,\dots,l_k\subset\mathbb CP^2: l_1\cap \cdots\cap l_k=\{p\}$, and $l\subset\mathbb CP^2: p\notin l$, as shown in Figure \ref{fig:mntR}.

\begin{figure}[h]
\begin{tikzpicture}[scale=1]
\begin{scope}[shift={(-3.3,0)}, scale=0.7]
\filldraw[black] (0,0) circle (2pt)node[above,scale=0.9]{$p$};
\draw (-.3,-.04) -- (2,.3)node[right,scale=0.9]{$l_2$};
\draw (-.2,.1) -- (2,-1)node[below right, scale=0.9]{$l_k$};
\draw (-.2,-.1) -- (2,1) node[above right, scale=0.9]{$l_1$};
\draw (1.7,1.2) -- (1.7,.1) (1.7,-.7)--(1.7,-1.2) node[below,scale=0.9]{$l$};
\draw[white] (1.7,-.2) circle (0pt) node[black]{$\vdots$};
\end{scope}

\draw[white] (-.5,0) circle (0.1pt) node[black,scale=0.7]{$\longrightarrow$};

\begin{scope}[shift={(.5,.7)}, scale=0.7]
\draw (1.7,-.2) -- (0,-.2) node[left,scale=0.9]{$0$};
\draw (0.3,-2) -- (0.3,-1.66) (0.3,-1)--(0.3,0);
\draw[white] (.3,-1.2) circle (0pt) node[black]{$\vdots$};
\filldraw[black] (0.3,0) circle (0pt)node[above,scale=0.9]{$-1$};
\draw (1.4,-2) -- (1.4,-1.66) (1.4,-1)--(1.4,0);
\draw[white] (1.4,-1.2) circle (0pt) node[black]{$\vdots$};
\filldraw[black] (1.4,0) circle (0pt)node[above,scale=0.9]{$+1$};
\draw (1.7,-.8) -- (0,-.8)node[left,scale=0.9]{$0$};
\draw (1.7,-1.8) -- (0,-1.8)node[left,scale=0.9]{$0$};
\end{scope}	

\draw[white] (4,0) circle (0.1pt) node[black]{$\subseteq\mathbb CP^2\#\overline{\mathbb CP}^2\ \ \ \rightarrow\cdots\rightarrow$};

\begin{scope}[shift={(7,0)}]
\draw[white] (-.3,0) circle (0.1pt) node[black]{$G$};
\draw (0,0) -- (.7,0)node[above,scale=0.9]{$-1$};
\draw (0,.4) -- (.7,.4)node[above,scale=0.9]{$-1$};
\draw[white] (.2,-.2) circle (0pt) node[black]{$\vdots$};
\draw (0,-.7) -- (.7,-.7)node[above,scale=0.9]{$-1$};
\draw[white] (1.3,0) circle (0.1pt) node[black]{$G^*$};
\end{scope}	
\end{tikzpicture}

\caption{We have that $R\setminus\mathring{\nu(M)}=P_{G}\cup P_{G^*}$.}
\label{fig:mntR}
\end{figure}
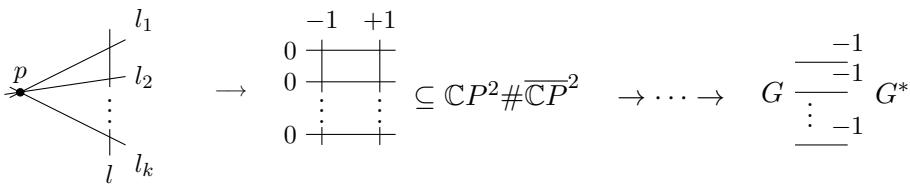

Now, we can extend the vector we read from surgery presentation on $P_{G}$ to a cohomology class $\mathfrak c$ on the whole $R$, which we describe through its Poincar\'e dual by $\text{PD}(\mathfrak c):=\alpha h+\sum_{i=1}^N\alpha_i e_i$ where $h$ and $e_i$ denote the standard generators of $H_2(R;\mathbb Z)$ and $\alpha,\alpha_i\in\{\pm1\}$. The characteristic vector $C$ is identified with the restriction $\mathfrak c\lvert_{P_{G^*}}$; note that there is a natural correspondence between the $\Spin^c$-structures $\mathfrak c\lvert_M$ and $\mathfrak c\lvert_{-M}$.

From \cite{OSz-fullpath} one has that $[C]\in\Imm\psi^*\subset\widehat{HF}(-M,\s_\xi)$ where $[C]$ is the full path of $C$ and $\psi^*:HF^-\rightarrow\widehat{HF}$ is the map induced by the projection. For the terminology about the full path algorithm we refer to \cite{OSz-fullpath,CM-negative,CM}. In \cite[Corollary 1.4 and Theorem 2.2]{Irena} it is shown that $C$ is well-defined, up to steps in the full-path, and that for $L$-spaces  it is identified with the contact invariant $\widehat c$ of the structure. 
Such a vector has been informally called \emph{magic $C$}, and we are going to refer to it in this way in the paper. 

\begin{teo}[Matkovi\v c]
 \label{teo:Irena}
 Let $M=M(e_0;r_1,r_2,r_3)$ with $e_0\in\{-1,0\}$ be a Seifert fibred $L$-space. Then a contact structure $\xi$ on $M$ is tight if and only if the full path $[C]$ ends correctly; in other words, when $\widehat c(\xi)\neq0$. Furthermore, two tight structures $\xi_1$ and $\xi_2$ on $M$ are contact isotopic if and only if their magic $C$'s are in the same full path; in other words, when $\widehat c(\xi_1)=\widehat c(\xi_2)$.
\end{teo}

In \cite{Irena} only $e_0=-1$ is addressed explicitly, because the classification of contact structures with $e_0=0$ was known before \cite{GLS=>}. However, magic $C$ can be defined analogously for $e_0=0$ and has the same properties; therefore, we include it in the statement of Theorem \ref{teo:Irena}.

Our goal now is to show that a similar result holds for any $M=M(-1;r_1,r_2,r_3)$. Since $M$ can either have indefinite or negative-definite standard graph, according to our work in \cite{CM} in the latter case we have $[C]\in\Ker\rho^*\subset\widehat{HF}(-M,\s_\xi)$ where $\rho^*:\widehat{HF}\rightarrow HF^+$ is induced by the inclusion. 

\begin{prop}
 \label{prop:zero}
 Let $M(-1;r_1,r_2,r_3)$ be a Seifert fibred space and $\xi$ be a contact structure presented as in \cite{LS}. We have that $\widehat c(\xi)\in\widehat{HF}(-M,\s_\xi)$ coincides with the magic $C$ associated to $\xi$; in particular, if the full path $[C]$ ends correctly then $\xi$ is tight and zero-twisting.     
\end{prop}
\begin{proof}
  From the discussion above we know that there exists a structure $\xi'$ on $M'=M(0;r_1,r_2,r_3)$, which is an $L$-space, obtained by removing one of the two standard unknots where we perform $(+1)$-surgery in the presentation of $\xi$. This means that $\xi$ is gotten from $\xi'$ by doing $(+1)$-surgery on a standard unknot $U$ in the complement of its contact surgery presentation.  

  From \cite{OSz-contact} the map induced by contact $(+1)$-surgery is $\widehat F_{-W,\mathfrak j}:\widehat{HF}(-M',\s_{\xi'})\rightarrow\widehat{HF}(-M,\s_\xi)$ and $\widehat F_{-W,\mathfrak j}(\widehat c(\xi'))=\widehat c(\xi)$, where $W$ is the cobordism defined by attaching a 2-handle along $U$ with framing zero and $\mathfrak j$ is the $\Spin^c$-structure on $W$ extending $\s_\xi$ and $\s_{\xi'}$ given by the coordinate zero. Applying a sequence of blow-ups and blow-downs as in Figure \ref{Structures2}, we have that $(W,\mathfrak j)$ is diffeomorphic to the 2-handle attachment along an unknot with framing $1$, linked to the central vertex of the standard graph of $M'$, with $\Spin^c$-structure corresponding to the coordinate $1$.

  Swapping the orientation of $W$, we have that the cobordism $-W$ is induced by a negative blow-up on the central vertex (whose framing is $-3$) of the standard graph of $-M'$. Let us denote by $P'$ the plumbing associated to this blown up graph for $-M$, and $\mathfrak j'\in\Spin^c(P')$ defined by the characteristic vector $(1,C')$; since $M'$ is an $L$-space, we can use Theorem \ref{teo:Irena} and then obtain that \[\widehat c(\xi)=\widehat F_{-W,\mathfrak j}[C']=\widehat F_{P',\mathfrak j'}(1)=\widehat F_{P,\mathfrak u}(1)\:,\] where $P$ is the plumbing associated to the standard graph of $-M$ and $\mathfrak u$ is defined by the characteristic vector $C'+e_1$. The last equality comes from a result of Ozsv\'ath and Szab\'o \cite[Proposition 2.5]{OSz-fullpath}, see also \cite[Subsection 2.1]{CM}, and the commutativity of the strict transform \cite[Section 5]{BP}. 

  In order to prove the first claim we just need to argue that the vector $C'+e_1$, obtained by increasing by one the first coordinate of $C'$, is precisely the magic $C$ associated to $\xi$. This follows easily from the algorithm for $C$ described above. The second claim is a consequence of \cite[Proposition 2.2]{CM-negative} and \cite[Theorem 1.1]{CM} which tell us that a full path ending correctly is identified with a non-zero element in $\widehat{HF}$.
\end{proof}
 
Applying Proposition \ref{prop:zero}, we can now prove the tightness of the structures $\xi_k$ for $k\leq98$ in Figure \ref{Structures}. We have that the vector of the rotation numbers is \[\begin{pmatrix}\begin{array}{c|ccc|cc|c}1 & -2 & -1 & -1 & 1 &-1 & 102-k\end{array}\end{pmatrix}\] which is characteristic for the standard graph of $S^3_k(T_{8,13})$ in Figure \ref{Structures2}. The algorithm described above, see \cite{Irena} for more details, gives that the magic $C$ associated to $\xi_k$ is \[C_k=\begin{pmatrix}\begin{array}{c|ccc|ccc|ccc} -2 & -1 & -1 & 0 & 2 & 1 & -2 & 2 & 0 & \cdots\end{array}\end{pmatrix}\] where $C_k$ ends with $102-k$ zeros.  

\begin{figure}[t]
\begin{tikzpicture}[scale=0.9]
    \tkzDefPoints{0/0/A, 1.5/1/B, 1.5/-1/D, 3/-1/E, 5/-1/F, 1.5/0/C, 3/1/G1, 4.5/1/G2, 3/0/H1, 4.5/0/H2}  
    \tkzDefPoints{3.5/-1/X, 4.5/-1/Y}\tkzDefPoints{3.9/-1/P, 4/-1/Q, 4.1/-1/R}
    \tkzDrawPoints[fill,black,size=1](P,Q,R)
    \tkzDrawSegment(A,B)\tkzDrawSegment(A,D)\tkzDrawSegment(E,D)\tkzDrawSegment(E,X)\tkzDrawSegment(Y,F)\tkzDrawSegment(A,C)\tkzDrawSegment(B,G1)\tkzDrawSegment(G2,G1)\tkzDrawSegment(C,H1)\tkzDrawSegment(H2,H1)
    \tkzDrawPoints[fill,black,size=5](A,B,C,D,E,F,G1,G2,H1,H2)
     \tkzLabelPoint[above left](A){$-2$} \tkzLabelPoint[above left](B){$-3$}\tkzLabelPoint[above right](C){$-2$}
     \tkzLabelPoint[below left](D){$-2$}\tkzLabelPoint[below left](E){$-2$}\tkzLabelPoint[below left](F){$-2$}\tkzLabelPoint[above left](G1){$-3$}\tkzLabelPoint[above left](G2){$-2$}\tkzLabelPoint[above right](H1){$-3$}\tkzLabelPoint[above right](H2){$-2$}
\end{tikzpicture}
     \caption{\smaller[1]{The standard graph $G^*$ of $-S^3_k(T_{8,13})\simeq M(-2;\frac{5}{8},\frac{5}{13},\frac{103-k}{104-k})$. There are $103-k$ vertices in the third leg, all with framing $-2$.}}
     \label{G}
\end{figure}
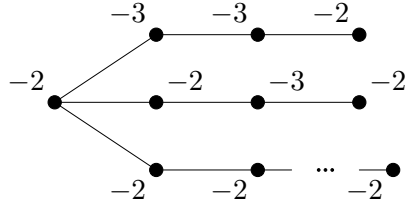 

\begin{lemma}
  \label{lemma:magic_C}
  The full path of the vector $C_k\in\emph{Char}(G^*,\s_\xi)$ ends correctly for every $k\leq98$, where here $G^*$ is the graph appearing in Figure \ref{G}.
\end{lemma}
\begin{proof}
 According to the definition in \cite{OSz-fullpath} and our notation in \cite{CM-negative}, the full path $[C_k]$ ends correctly if and only if following the path from the vectors $C_k$ and $-C_{k}$ leads to a vector $(a_1,...,a_{|G|})$ such that $m(i)\leq a_i\leq -m(i)-2$ for each $i=1,...,|G^*|$, where $m(i)$ is the framing of the $i$-th vertex of $G^*$. We recall that given a characteristic vector $V$ we take a step from $V$ to $V+2Q_*e_i$ when $v_i=-m(i)$, where $Q_*$ is the intersection matrix of $G^*$. Let us start from $-C_k$: \[-C_k=\begin{pmatrix}\begin{array}{c|ccc|ccc|ccc} 2 & 1 & 1 & 0 & -2 & -1 & 2 & -2 & 0 & \cdots\end{array}\end{pmatrix}\overset{7}{\longrightarrow}\] \[\begin{pmatrix}\begin{array}{c|ccc|ccc|ccc} 2 & 1 & 1 & 0 & -2 & 1 & -2 & -2 & 0 & \cdots\end{array}\end{pmatrix}\overset{1}{\longrightarrow}\] \[\begin{pmatrix}\begin{array}{c|ccc|ccc|ccc} -2 & 3 & 1 & 0 & 0 & 1 & -2 & 0 & 0 & \cdots\end{array}\end{pmatrix}\overset{2,3,4}{\longrightarrow}\] \[\begin{pmatrix}\begin{array}{c|ccc|ccc|ccc} 0 & -1 & -1 & -2 & 0 & 1 & -2 & 0 & 0 & \cdots\end{array}\end{pmatrix}\:.\]
 The dots always represent a string of zeros. Let us continue with $C_k$ itself: \[C_k=\begin{pmatrix}\begin{array}{c|ccc|ccc|ccc} -2 & -1 & -1 & 0 & 2 & 1 & -2 & 2 & 0 & \cdots\end{array}\end{pmatrix}\overset{5,6}{\longrightarrow}\] \[\begin{pmatrix}\begin{array}{c|ccc|ccc|ccc} 0 & -1 & -1 & 0 & 0 & -3 & 0 & 2 & 0 & \cdots\end{array}\end{pmatrix}\overset{3\text{rd leg}}{\longrightarrow}\] \[\begin{pmatrix}\begin{array}{c|ccc|ccc|ccc} 2 & -1 & -1 & 0 & 0 & -3 & 0 & 0 & \cdots & -2\end{array}\end{pmatrix}\overset{1}{\longrightarrow}\] \[\begin{pmatrix}\begin{array}{c|ccc|ccc|ccc} -2 & 1 & -1 & 0 & 2 & -3 & 0 & 2 & \cdots & -2\end{array}\end{pmatrix}\overset{5}{\longrightarrow}\] \[\begin{pmatrix}\begin{array}{c|ccc|ccc|ccc} 0 & 1 & -1 & 0 & -2 & -1 & 0 & 2 & \cdots & -2\end{array}\end{pmatrix}\overset{3\text{rd leg}}{\longrightarrow}\] \[\begin{pmatrix}\begin{array}{c|ccc|ccc|ccc} 2 & 1 & -1 & 0 & -2 & -1 & 0 & \cdots & -2 & 0 \end{array}\end{pmatrix}\overset{1}{\longrightarrow}\] \[\setcounter{MaxMatrixCols}{30}\begin{pmatrix}\begin{array}{c|ccc|ccc|cccc} -2 & 3 & -1 & 0 & 0 & -1 & 0 & 2 & \cdots & -2 & 0 \end{array}\end{pmatrix}\overset{2}{\longrightarrow}\] \[\setcounter{MaxMatrixCols}{30}\begin{pmatrix}\begin{array}{c|ccc|ccc|cccc} 0 & -3 & 1 & 0 & 0 & -1 & 0 & 2 & \cdots & -2 & 0 \end{array}\end{pmatrix}\overset{3\text{rd leg}}{\longrightarrow}\]
 \[\setcounter{MaxMatrixCols}{30}\begin{pmatrix}\begin{array}{c|ccc|ccc|cccc} 2 & -3 & 1 & 0 & 0 & -1 & 0 & \cdots & -2 & 0 & 0 \end{array}\end{pmatrix}\overset{1}{\longrightarrow}\]
 \[\setcounter{MaxMatrixCols}{30}\begin{pmatrix}\begin{array}{c|ccc|ccc|ccccc} -2 & -1 & 1 & 0 & 2 & -1 & 0 & 2 & \cdots & -2 & 0 & 0 \end{array}\end{pmatrix}\overset{5}{\longrightarrow}\]
 \[\setcounter{MaxMatrixCols}{30}\begin{pmatrix}\begin{array}{c|ccc|ccc|ccccc} 0 & -1 & 1 & 0 & -2 & 1 & 0 & 2 & \cdots & -2 & 0 & 0 \end{array}\end{pmatrix}\overset{3\text{rd leg}}{\longrightarrow}\]
 \[\setcounter{MaxMatrixCols}{30}\begin{pmatrix}\begin{array}{c|ccc|ccc|ccccc} 2 & -1 & 1 & 0 & -2 & 1 & 0 & \cdots & -2 & 0 & 0 & 0 \end{array}\end{pmatrix}\overset{1}{\longrightarrow}\]
 \[\setcounter{MaxMatrixCols}{30}\begin{pmatrix}\begin{array}{c|ccc|ccc|cccccc} -2 & 1 & 1 & 0 & 0 & 1 & 0 & 2 & \cdots & -2 & 0 & 0 & 0 \end{array}\end{pmatrix}\overset{3\text{rd leg}}{\longrightarrow}\]
 \[\begin{pmatrix}\begin{array}{c|ccc|ccc|cccccc} 0 & 1 & 1 & 0 & 0 & 1 & 0 & \cdots & -2 & 0 & 0 & 0 & 0 \end{array}\end{pmatrix}\:.\]
\end{proof}

\begin{proof}[Proof of Theorem \ref{teo:main} and Corollary \ref{cor:main}]
 The corollary follows immediately from Theorems \ref{teo:old} and \ref{teo:main}. The fact that $\widehat c(\xi_k)\neq0$ is a consequence of Proposition \ref{prop:zero} and Lemma \ref{lemma:magic_C}. We now show Property 1): we need to distinguish between the cases where the standard graph of $S^3_k(T_{8,13})$ is either negative-definite (when $k\leq-1$) or indefinite (when $1\leq k\leq35$). 

 Let us assume that $k\leq-1$; then we have that $c^+(\xi_k)=\rho^*(\widehat c(\xi_k))=0$, see \cite[Equation 2.2 and Theorem 2.1]{CM} (this holds true also when $k=0$ by \cite[Lemma 2.3]{CM} as $C_0$ is non-torsion on the boundary). Let us assume that $1\leq k\leq35$, the manifold $-S^3_k(T_{8,13})$ is presented by the graph $G^*$ in Figure \ref{G} which in this case is negative-definite with one bad vertex; then it follows from the main result in \cite{OSz-fullpath}, see also \cite{CM-negative} for details, that $c^+(\xi_k)\in HF^+(-S^3_k(T_{8,13}),\s_{\xi_k})$ is non-zero if and only if $c^+(\xi_k)=\Theta^+_{\s_{\xi_k}}$, and this happens exactly when \[M(C_k)=d:=d(-S^3_k(T_{8,13}),\s_{\xi_k})\] where $d$ is the correction term. We recall that for a 3-manifold $(Y,\s)$ we define $\Theta_\s^+$ as the only non-zero class in $\Ker U\cap U^n\cdot HF^+(Y,\s)$ for any $n\geq0$.
 
 Denote by \[\mathcal S=\{V\in\text{Char}(G,\:\s_{\xi_k})\:|\:[V]\text{ ends correctly}\}\] where $G$ is on the right in Figure \ref{Structures2}; using standard linear algebra, when $1\leq k\leq35$ we can write \[M(C_k)=\dfrac{-k^2+153k-5184}{4k}<-\frac{15}{2}<-\dfrac{V_k^TQ_{G}^{-1}V_k+|G|-6}{4}\leq-\min_{V\in\mathcal S}M(V)<d\] where the full path of \[V_k=\begin{pmatrix}1 & 0 & -1 & -1 & -1 & -1 & 52-k\left(1+2\left\lfloor\frac{36}{k}\right\rfloor\right)\end{pmatrix}\in\text{Char}(G,\s_{\xi_k})\] ends correctly, the formula for $M(V)$ comes from \cite[Equation 3.1]{CM}, while the last inequality from \cite[Proposition 2.2]{CM-negative} and the fact that $d-M(V)$ is odd \cite[Theorem 2.1]{CM}.

 Property 2) follows immediately from above, because we show precisely that in these cases the $d_3$-invariant is not the same as the corresponding correction term.
 Regarding Property 3): we know that these manifolds are not $L$-spaces for any $k<83$, and the structures we are considering are not fillable because $S_k^3(T_{8,13})$ carries no such structure from the criterion in \cite[Theorem 1.1]{Irena(f)}; namely, that the standard graph $G$ does not contain a subgraph with two complementary legs (this obstruction also holds when $k=0$).
\end{proof}

Note that $\xi_k$ is certainly not the only zero-twisting tight structure on $S^3_k(T_{8,13})$; in fact, the same is true for the conjugate $\overline\xi_k$, which is not contact isotopic to $\xi_k$ because its magic $C$ is in the full path $[-C_k]$, and we immediately see that $[C_k]\neq[-C_k]$ thus $\widehat c(\xi_k)\neq\widehat c(\overline\xi_k)$.

\end{document}